\documentclass[12pt]{amsart}

\usepackage{fullpage}
\usepackage{latexsym,amssymb,amsfonts,amsmath,mathrsfs,float}
\usepackage[font=small,labelfont=bf, width=.618\textwidth]{caption} 
\usepackage{xcolor}
\usepackage{tikz, graphics, enumerate}
\usepackage{hyperref}

  \renewcommand{\Pr}{\mbox{\rm Pr}}	
  
  \newcommand{\E}{\mathbb{E}}

  \newcommand{\R}{\mathbb{R}} 

  \newcommand{\st}{:\,} 

  \newcommand{\eps}{\varepsilon}

  \DeclareMathOperator{\diag}{diag}

  \newcommand{\beq}{\begin{equation}}
  \newcommand{\eeq}{\end{equation}}
  \newcommand{\beqn}{\begin{equation*}}
  \newcommand{\eeqn}{\end{equation*}}
  \newcommand{\beqr}{\begin{eqnarray}}
  \newcommand{\eeqr}{\end{eqnarray}}
  \newcommand{\beqrn}{\begin{eqnarray*}}
  \newcommand{\eeqrn}{\end{eqnarray*}}
  \newcommand{\bmline}{\begin{multline}}
  \newcommand{\emline}{\end{multline}}
  \newcommand{\bmlinen}{\begin{multline*}}
  \newcommand{\emlinen}{\end{multline*}}

  \theoremstyle{plain}
  \newtheorem{theorem}{Theorem}[section]
  \newtheorem{lemma}[theorem]{Lemma}
  
  \newtheorem{claim}[theorem]{Claim}
  \newtheorem{corollary}[theorem]{Corollary}
  
  \theoremstyle{definition}
  \newtheorem{definition}[theorem]{Definition}

  \theoremstyle{remark}
  
  \renewenvironment{proof}[1][]{
    	\begin{trivlist}
     	\item[\hspace{\labelsep}{\em\noindent Proof#1:\/}]}
     	{{\hfill$\Box$}
    	\end{trivlist}
  }
  
  \makeatletter
  \newtheorem*{rep@theorem}{\rep@title}
  \newcommand{\newreptheorem}[2]{%
  \newenvironment{rep#1}[1]{%
  \def\rep@title{#2 \ref{##1}}%
  \begin{rep@theorem}}%
  {\end{rep@theorem}}}
  \makeatother

  \newreptheorem{lemma}{Lemma}
  \newreptheorem{theorem}{Theorem}
  \newreptheorem{corollary}{Corollary}
  \newreptheorem{proposition}{Proposition}
  \newreptheorem{conjecture}{Conjecture}
  \newreptheorem{example}{Example}

\begin{document}
\title{A Hoeffding inequality for Markov chains}
\author{
       Shravas Rao}
\address{Courant Institute, New York University, 251 Mercer Street, New York NY 10012, USA}
\email{rao@cims.nyu.edu}
\thanks{This material is based upon work supported by the National Science Foundation Graduate Research Fellowship Program under Grant No. DGE-1342536.}
\date{\today}

\maketitle
\begin{abstract}
We prove deviation bounds for the random variable $\sum_{i=1}^{n} f_i(Y_i)$ in which~$\{Y_i\}_{i=1}^{\infty}$ is a 
Markov chain with stationary distribution and state space $[N]$, and $f_i: [N] \rightarrow [-a_i, a_i]$.
Our bound improves upon previously known bounds in that the dependence is on $\sqrt{a_1^2+\cdots+a_n^2}$ rather than $\max_{i}\{a_i\}\sqrt{n}.$
We also prove deviation bounds for certain types of sums of vector--valued random variables obtained from a Markov chain in a similar manner.
One application includes bounding the expected value of the Schatten~$\infty$-norm of a random matrix whose entries are obtained from a Markov chain.
\end{abstract}

\section{Introduction}

Consider a 
Markov chain $\{Y_i\}_{i=1}^{\infty}$ with state space $[N]$, transition matrix $A$, and stationary distribution $\pi$ such that $Y_1$ is distributed as $\pi$.
Let $E_{\pi}$ be the associated averaging operator defined by $(E_{\pi})_{ij} = \pi_j$, so that for $v \in \R^N$ $E_{\pi}v = \E_{\pi}[v]\mathbf{1}$ where $\mathbf{1}$ is the vector whose entries are all $1$.

In the case that the $Y_i$ are independent, that is $A = E_{\pi}$, then it is well known (see~\cite{H63}) that for functions $f_1, \ldots, f_n: [N] \rightarrow [-1, 1]$ with $\E[f_i(Y_i)] = 0$ for all $i$ and $u \geq 0$, that
\begin{equation}\label{eq:grandpa}
\Pr[|f_1(Y_1)+\cdots+f_n(Y_n)| \geq u\sqrt{n}] \leq 2\exp\left(-u^2/2\right).
\end{equation}
Gillman generalized Eq.~\eqref{eq:grandpa} to all 
Markov chains with a stationary distribution, in terms of the quantity $\lambda = \|A-E_{\pi}\|_{L_2(\pi) \rightarrow L_2(\pi)}$ in the case $f_1 = \cdots = f_n$~\cite{G98}.
These bounds were refined in a long series of work including~\cite{D95, K97, L98, Wagner08, LCP04, H08, ChungLLM12, HazlaH15, P15, NaorRR17, RR17}.
We state the following version due to Healy~\cite{H08}, which handles the case in which the $f_i$ are not necessarily equal.
\begin{equation}\label{eq:h08}
\Pr[|f_1(Y_1)+\cdots+f_n(Y_n)| \geq u\sqrt{n}] \leq 2\exp\left(\frac{-u^2(1-\lambda)}{4}\right).
\end{equation}

Back in the case of independent random variables, Hoeffding generalized Eq.~\eqref{eq:grandpa} to the case when the function $f_i$ has range $[-a_i, a_i]$, obtaining the following bound~\cite{H63}.
\begin{equation}\label{eq:hoeffding}
\Pr\left[|f_1(Y_1)+\cdots+f_n(Y_n)|  \geq u\left(\sum_{i=1}^n a_i^2\right)^{1/2} \right]
\leq
2\exp(-u^2/2).
\end{equation}
In this work, we generalize Eq.~\eqref{eq:hoeffding} to 
Markov chains with a stationary distribution.
In particular, we prove the following.

\begin{theorem}\label{cor:tailbound} 
Let $\{Y_i\}_{i=1}^{\infty}$ be a stationary 
Markov chain with state space $[N]$, transition matrix $A$, stationary probability measure $\pi$, and averaging operator $E_{\pi}$, so that $Y_1$ is distributed according to $\pi$.
Let $\lambda = \|A-E_{\pi}\|_{L_2(\pi) \rightarrow L_2(\pi)}$ and let $f_1, \ldots, f_n: [N] \rightarrow \R$ so that $\E[f_i(Y_i)] = 0$ for all $i$ and $|f_i(v)| \leq a_i$ for all $v \in [N]$ and all $i$.
Then for $u \geq 0$,
\[
\Pr\left[|f_1(Y_1)+\cdots+f_n(Y_n)| \geq u\left(\sum_{i=1}^n a_i^2\right)^{1/2} \right]
\leq
2\exp(-u^2(1-\lambda)/(64e)).
\]
\end{theorem}

One interpretation of Theorem~\ref{cor:tailbound} is that for a Markov chain $\{Y_i\}_{i=1}^{\infty}$ and functions $f_1, \ldots, f_n: [N] \rightarrow [-1, 1]$, the random vector $(f_1(Y_1), \ldots, f_n(Y_n))$ is sub--gaussian.

We remark that the dependence on $\lambda$ in both Eq.~\eqref{eq:h08} and Theorem~\ref{cor:tailbound} is optimal, as shown in~\cite{LCP04} which considered the case that the $f_i$ are equal.
In particular, one can consider the Markov chain on two states with the transition matrix
\[
\begin{bmatrix}
\frac{1+\lambda}{2} & \frac{1-\lambda}{2} \\
\frac{1-\lambda}{2} & \frac{1+\lambda}{2}
\end{bmatrix}
\]
so that $f_i(1) = 1$ and $f_i(2) = -1$ for all $i$.
Intuitively, the random variable $f_1(Y_1)+\cdots+f_n(Y_n)$ is similar to the sum of $n(1-\lambda)$ random variables that are close to $1/(1-\lambda)$ or close to $-1/(1-\lambda)$, both with equal probability.

We also remark that Theorem~\ref{cor:tailbound} holds even for non-reversible Markov chains, continuing the work of~\cite{ChungLLM12} who were the first to consider this setting.
It is possible, if the Markov chain is not reversible, for $\|A-E_{\pi}\|_{L_2(\pi) \rightarrow L_2(\pi)}$ to be greater than $1$, and thus the bound in Theorem~\ref{cor:tailbound} is trivial.

\subsection{Extension to vector--valued random variables}

Recently, much attention has been paid to tail bounds for sums of vector--valued random variables.
Naor~\cite{N12} obtained tail bounds for sums of random variables from a Banach space satisfying certain properties.
Before stating the corresponding theorem, we define a quantity called the modulus of uniform smoothness.

\begin{definition}
The modulus of uniform smoothness of a Banach space $(X, \| \cdot \|)$ is \[\rho_X(\tau) = \sup\left\{\frac{\|x+\tau y\|+\|x - \tau y\|}{2}-1 : x, y \in X, \|x\| = \|y\| = 1\right\}.\]
\end{definition}

Let $(X, \|\cdot\|)$ be a Banach space so that $\rho_X(\tau) \leq s \tau^2$ for some $s$ and all $\tau > 0$.
When the elements of the Markov chain are independent, for $f_i: [N] \rightarrow \{x \in X \st \|x\| \leq a_i\}$ and such that $\E[f_i(Y_i)] = 0$, it was shown that
\begin{equation}\label{thm:banachspacech}
\Pr\left[\|f_1(Y_1)+\cdots+f_n(Y_n)\| \geq u \left(\sum_{i=1}^n a_i^2\right)^{1/2}\right] \leq \exp\left(s+2-{cu^2}\right)
\end{equation}
for some universal constant $c$.


We extend Theorem~\ref{cor:tailbound} to random variables from a fixed Banach space as follows.
We stress that the setting in the following theorem is more limited than that of Eq.~\eqref{thm:banachspacech}.
In particular we only allow random variables of the form $f(Y_i)X_i$ in which $f(Y_i)$ is a random scalar and $X_i$ is a fixed element from the Banach space.

\begin{theorem}\label{thm:app}
Let $(X, \|\cdot\|)$ be a Banach space, and let $X_1, \ldots, X_n \in X$.
Let $\{Y_i\}_{i=1}^{\infty}$ be a stationary 
Markov chain with state space $[N]$, transition matrix $A$, stationary probability measure $\pi$, and averaging operator $E_{\pi}$, so that $Y_1$ is distributed according to $\pi$.
Let $\lambda = \|A-E_{\pi}\|_{L_2(\pi) \rightarrow L_2(\pi)}$, and let $f_1, \ldots, f_n: [N] \rightarrow [-1, 1]$ be such that $\E[f_i(Y_i)] = 0$ for all $i$.
Then there exist universal constants $C$ and $L$, such that for any $u \geq 0$,
\[
\Pr\left[\left\|f_1(Y_1)X_1+\cdots+f_n(Y_n)X_n\right\| \geq 
{uC\E[\|g_1X_1 + \cdots + g_n X_n\|]}\right]
\leq
L\exp(-Cu^2(1-\lambda))
\]
where $g_1, \ldots, g_n \sim \mathcal{N}(0, 1)$ are independent standard Gaussian random variables.
\end{theorem}

Note that Eq.~\eqref{thm:banachspacech} implies that $\E[\|g_1X_1 + \cdots + g_n X_n\|] \leq C\sqrt{s (\|X_1\|^2+\cdots+\|X_n\|^2)}$ for some constant $C$.
This follows from the fact that the distribution of the normalized sum of independent Rademacher random variables approaches that of a Gaussian, in the limit.
Thus for Banach spaces that satisfy $\rho_X(\tau) \leq s\tau^2$, we also have the bound
\[
\Pr\left[\left\|f_1(Y_1)X_1+\cdots+f_n(Y_n)X_n\right\| \geq 
{uC\sqrt{s (\|X_1\|^2+\cdots+\|X_n\|^2)}}\right]
\leq
L\exp(-Cu^2(1-\lambda))
\]

\subsubsection{Bounds on the Schatten $\infty$-norm of a random matrix}

As an application, we are able to generalize bounds on the Schatten $\infty$-norm of a matrix with independent entries to matrices whose entries are obtained from a 
Markov chain with stationary distribution.

Let $\mathcal{I} \subseteq [d] \times [d]$ be the set of pairs $(i, j)$ such that $i \leq j$, and let $B = (b_{i, j}) \in \R^{d \times d}$ be a symmetric matrix with positive entries.
Let $X \in \R^{d \times d}$ be the random symmetric matrix whose entries are 
\[
X_{i, j} = 
\begin{cases}
\eps_{i, j}b_{i, j} & \text{ if } (i, j) \in \mathcal{I}\\
\eps_{j, i} b_{i, j} & \text{ otherwise}
\end{cases} 
\]
where $\eps_{i, j}$ are independent Rademacher random variables.
Then it was shown in~\cite{BvH16} that
\begin{equation}\label{eq:bvh}
\E[\|X\|_{S_{\infty}}] \leq \min\left\{C (\sigma + \sigma_*\sqrt{\log d}), \|B\|_{S_{\infty}}\right\}
\end{equation}
for some absolute constant $C$, where
\begin{equation}\label{eq:sigmadef}
\sigma = \max_{i} \sqrt{\sum_j b_{i, j}^2} \text{ and } \sigma_* = \max_{i, j} |b_{i, j}|.
\end{equation}
We generalize Eq.~\eqref{eq:bvh} to 
Markov chains with a stationary distribution.
In particular, we obtain a similar bound in terms of $\lambda = \|A-E_{\pi}\|_{L_2(\pi) \rightarrow L_2(\pi)}$ on the Schatten $\infty$-norm of a matrix whose entries are chosen in the following manner.
We start by choosing an arbitrary permutation of the entries in the diagonal and upper triangular part of the matrix.
Then we fill in the entries according to the order given by the permutation, using the values given by the Markov chain. 
Finally we fill in the entries in the lower triangular part of the matrix, so that the matrix is symmetric.
The case that the transition matrix is $A = E_{\pi}$ corresponds to choosing the entries of the diagonal and upper triangular part of the matrix independently, as in~\cite{BvH16}.

\begin{corollary}\label{cor:apptobvh}
Let $\{Y_i\}_{i=1}^{\infty}$ be a stationary 
Markov chain with state space $[N]$, transition matrix $A$, stationary probability measure $\pi$, and averaging operator $E_{\pi}$, so that $Y_1$ is distributed according to $\pi$.
Let $\lambda = \|A-E_{\pi}\|_{L_2(\pi) \rightarrow L_2(\pi)}$, let $f: V \rightarrow [-1, 1]$ be such that $\E[f(Y_i)] = 0$, and let $B \in \R^{d \times d}$ be a symmetric $d \times d$ matrix with positive entries.
For any injective function $\omega: \mathcal{I} \rightarrow \{1, 2, \ldots, (d^2+d)/2\} $, let $X$ be the symmetric matrix defined by 
\[
X_{i, j} = 
\begin{cases}
f(Y_{\omega(i, j)})b_{i, j} & \text{ if } (i, j) \in \mathcal{I}\\
f(Y_{\omega(j, i)}) b_{j, i} & \text{ otherwise}
\end{cases} 
\]
Then, 
\[
\E[\|X\|_{S_{\infty}}] \leq \min\left\{\frac{C}{\sqrt{1-\lambda}} (\sigma + \sigma_*\sqrt{\log d}), \|B\|_{S_{\infty}}\right\},
\]
for some absolute constant $C$, where $\sigma$ and $\sigma_*$ are defined as in Eq.~\eqref{eq:sigmadef}.
\end{corollary}

\subsection{Related Work}

In recent independent work by Fan, Jiang, and Sun~\cite{FJS18}, a Hoeffding bound for general Markov chains was also given.
Their bound is sharper, and in particular the constant $64e$ can be replaced by $2$ after replacing $1-\lambda$ by $(1-\lambda)/(1+\lambda)$.
However, our proof is arguably somewhat simpler.

In work by Garg, Lee, Song and Srivastava~\cite{GLSS17}, a version of Eq.~\eqref{thm:banachspacech} was proved for Markov chains when the Banach space is the set of $d \times d$ matrices under the Schatten $\infty$-norm, generalizing a result first shown by Ahlswede and Winter~\cite{AW06} (see also the monograph by Tropp~\cite{T15}).
Note that the Schatten $\infty$-norm of a $d \times d$ matrix is up to constant factors equal to the Schatten $\log(d)$-norm of that matrix, and the modulus of uniform smoothness of the set of $d \times d$ matrices under the Schatten $p$-norm is $O(p \tau^2)$.
Thus in this case the left-hand side of Eq.~\eqref{thm:banachspacech} is bounded above by $d\exp(c-cu^2)$.
Garg, Lee, Song and Srivastava showed that for a Markov chain $\{Y_i\}_{i=1}^{\infty}$ and functions $f_i: [N] \rightarrow \{x \in \R^{d \times d} \st \|x\|_{S_{\infty}} \leq 1\}$ such that $\E[f_i(Y_i)] = 0$,
\begin{equation*}
\Pr[\|f_1(Y_1)+\cdots+f_n(Y_n)\|_{S_{\infty}} \leq u\sqrt{n}] \leq 2d \exp\left(-c(1-\lambda)u^2\right)
\end{equation*}


\section{Preliminaries}

Given vectors $v, \pi \in \R^{N}$ so that $\pi$ has positive entries, (typically $\pi$ will be a distribution over $[N]$), let
\[
\|v\|_{L_{p}(\pi)}^p = \sum_{i=1}^N \pi_i |v_i|^p.
\]
We define the inner product for two vectors $u, v \in \R^{N}$ and $\pi \in \R^N$ with positive entries to be
\[
\langle u, v\rangle_{L_2(\pi)} \sum_{i=1}^{N} \pi_i u_i v_i.
\]
Additionally, we let the operator norm of a matrix $A \in \R^{N \times N}$ be defined as 
\[\|A\|_{{L_p(\pi)} \rightarrow {L_q(\pi)}} = \max_{v: \|v\|_{L_p(\pi)} = 1} \|Av \|_{L_q(\pi)}.\]
We will use $\ell_p$ in place of $L_p(\mathbf{1})$ where $\mathbf{1}$ is the vector whose entries are all $1$.

The Schatten $p$-norm of a matrix $A \in \R^{N \times N}$ is defined to be
\[
\|A\|_{S_p}^p = \sum_{i=1}^N s_i^p
\]
where $s_1, \ldots, s_{N}$ are the singular values of $A$.

For a vector $v$, we let $\diag(v)$ be the diagonal matrix where $\diag(v)_{i, i} = v_i$.

Let $A$ be a stochastic matrix, and let $\pi$ be a stationary distribution for$A$.
We let $(E_{\pi})_{ij} = \pi_j$ be the averaging operator on $L_{\infty}(\pi) \rightarrow L_{\infty}(\pi)$.
Note that $E_{\pi}$ is also stochastic, and that $E_{\pi}A = AE_{\pi} = E_{\pi}^2 = E_{\pi}$.

%

The following simple claim bounds $\|T\|_{L_2(\pi) \rightarrow L_2(\pi)}$ for a matrix $T$ in terms of $\|T\|_{L_1(\pi) \rightarrow L_1(\pi)}$ and $\|T\|_{L_{\infty}(\pi) \rightarrow L_{\infty}(\pi)}$.
This can be viewed as a special case of interpolation of matrix norms.

\begin{claim}\label{thm:mlrt}
For any matrix $T$,
\[
\|T\|_{L_2(\pi) \rightarrow L_2(\pi)}^2 \leq \|T\|_{L_1(\pi) \rightarrow L_1(\pi)}\|T\|_{L_{\infty}(\pi) \rightarrow L_{\infty}(\pi)}.
\]
\end{claim}
\begin{proof}
For all $x, \pi \in \R^n$ so that $\pi$ has positive entries,
\begin{multline*}
\|Tx\|_{L_2(\pi) \rightarrow L_2(\pi)}^2 = \sum_{i=1}^n\pi_i\left(\sum_{j=1}^n T_{ij}x_{j} \right)^2 \leq \sum_{i=1}^n \pi_i\left(\sum_{j=1}^n |T_{ij}|\right)\left(\sum_{j=1}^n|T_{ij}|x_j^2 \right) 
\\
\leq \|T\|_{L_{\infty}(\pi) \rightarrow L_{\infty}(\pi)} \|T(x \circ x)\|_{L_1(\pi) \rightarrow L_1(\pi)} \leq \|T\|_{L_{\infty}(\pi) \rightarrow L_{\infty}(\pi)}\|T\|_{L_1(\pi) \rightarrow L_1(\pi)}\|x\|_{L_2(\pi)}^2
\end{multline*}
where the first inequality follows by Cauchy-Schwarz, and $\circ$ denotes entrywise product.\
\end{proof}

\section{Proof of Theorem~\ref{cor:tailbound}}

To prove Theorem~\ref{cor:tailbound}, we follow the strategy of bounding the $q$th moment for some even integer $q$, and using Markov's inequality to obtain a tail bound.
We start by expanding $(f_1(Y_1)+\cdots+f_n(Y_n))^q$ into a sum of monomials.

The following lemma bounds the expectation of monomials in the $f_i(Y_i)$.
The statement is similar to Lemma 3.3 in~\cite{RR17}.
Most of the proof is the same and is deferred to the appendix.
Let $S_{q-1} \subset \{0, 1\}^{q-1}$ be the set of strings with no consecutive $0$'s and so that $s_1, s_{q-1} = 1$ for all $s \in S_{q-1}$.

\begin{lemma}\label{lem:monomial}
Let $\{Y_i\}_{i=1}^{\infty}$ be a stationary 
Markov chain with state space $[N]$, transition matrix $A$, stationary probability measure $\pi$, and averaging operator $E_{\pi}$, so that $Y_1$ is distributed according to $\pi$.
Let $\lambda = \|A-E_{\pi}\|_{L_2(\pi) \rightarrow L_2(\pi)}$ and let $f_1, \ldots, f_n: [N] \rightarrow \R$ so that $\E[f_i(Y_i)] = 0$ for all $i$ and $|f_i(v)| \leq a_i$ for all $v \in [N]$ and all $i$.
For all $q$, and $w \in [n]^q$ such that $w_1 \leq w_2 \leq \cdots \leq w_q$
\[
\E[f_{w_1}(Y_{w_1}) f_{w_2}(Y_{w_2}) \cdots f_{w_q}(Y_{w_q})]
\leq
a_{w_1}a_{w_2}\cdots a_{w_q}\sum_{s \in S_{q-1}}
\left(\prod_{i: s_i = 1} \lambda^{w_{i+1}-w_i} \right).
\]
\end{lemma}
\begin{proof}
We apply Lemma~\ref{lem:holderapplication}, letting $k = q-1$, $u_i(v) = f_{w_i}(v)$ for all $v \in [N]$, and $T_i = A^{w_{i+1}-w_i} - E_{\pi}$.
Note that for all $k \geq 0$,
\[
A^k-E_{\pi} = A^{k}-A^{k-1}E_{\pi}-E_{\pi}A+E_{\pi}^2 = (A^{k-1}-E_{\pi})(A-E_{\pi}) = (A-E_{\pi})^k.
\]
The lemma follows by noting that $\|u_i\|_{L_\infty(\pi)} \leq a_{w_i}$ and $\|T_i\|_{L_2(\pi) \rightarrow L_2(\pi)} \leq \lambda^{w_{i+1}-w_i}$
\end{proof}

We obtain the following bound on the moments of $f_1(Y_1)+\cdots+f_n(Y_n)$. 

\begin{theorem}\label{lem:partition}
Let $\{Y_i\}_{i=1}^{\infty}$ be a stationary 
Markov chain with state space $[N]$, transition matrix $A$, stationary probability measure $\pi$, and averaging operator $E_{\pi}$, so that $Y_1$ is distributed according to $\pi$.
Let $\lambda = \|A-E_{\pi}\|_{L_2(\pi) \rightarrow L_2(\pi)}$ be less than $1$, and let $f_1, \ldots, f_n: [N] \rightarrow \R$ so that $\E[f_i(Y_i)] = 0$ for all $i$ and $|f_i(v)| \leq a_i$ for all $v \in [N]$ and all $i$.
Then for even $q$,
\[
\E[(f_1(Y_1)+\cdots+f_n(Y_n))^q]
\leq
4^q(q/2)!\left(\frac{1}{1-\lambda}\right)^{q/2} \left(\sum_{i=1}^n a_i^2 \right)^{q/2}.
\]
\end{theorem}
\begin{proof}
Let $\sigma: [n]^q \rightarrow [n]^q$ be the function where $\sigma(w)$ is the sorted list of coordinates of $w$ in non-decreasing order.
Then by Lemma~\ref{lem:monomial},
\begin{align}
\E[(f_1(Y_1)+\cdots+f_n(Y_n))^q]
&=
\sum_{w \in [n]^q} \E[f_{w_1}(Y_{w_1})f_{w_2}(Y_{w_2})\cdots f_{w_q}(Y_{w_q})] \nonumber\\
&\leq
\sum_{w \in [n]^q} a_{w_1}a_{w_2}\cdots a_{w_q}\sum_{s \in S_{q-1}}
\left(\prod_{i: s_i = 1} \lambda^{\sigma(w)_{i+1}-\sigma(w)_i} \right). \label{eq:monomintermediate}
\end{align}
Let $\binom{[q]}{q/2}$ denote the collection of subsets of $[q]$ of size exactly $q/2$.
For each subset $\mathcal{I}  \in \binom{[q]}{q/2}$, let $W_{\mathcal{I}} \subset [n]^q$ be the set of all vectors $w$ such that 
for each $j \in [n]$,
\[
\left|\{i \st i \in \mathcal{I} \text{ and } w_i = j\}\right| = \left|\{i \st i \in \{1, 3, 5, \ldots, q-1\} \text{ and } \sigma(w)_i = j\}\right|,
\]
i.e. the multi-set $\bigcup_{i \in \mathcal{I}} \{w_i\}$ is equal to the multi-set $\{\sigma(w)_1, \sigma(w)_3, \sigma(w)_5, \ldots, \sigma(w)_{q-1}\}$.
Let $w_{\mathcal{I}}, w_{[q] \backslash \mathcal{I}} \in [n]^{q/2}$ be the restriction of $w$ to the coordinates in $\mathcal{I}$ and $[q] \backslash \mathcal{I}$ respectively.
Additionally, for each $\mathcal{I}  \in \binom{[q]}{q/2}$ and $s \in S_{q-1}$, let $T_{\mathcal{I}, s}$ be the $n^{q/2} \times n^{q/2}$ matrix defined as follows.  
For each $w \in [n]^q$, the entry in the $w_{\mathcal{I}}$th row and $w_{[q] \backslash \mathcal{I}}$th column of $T_{\mathcal{I}, s}$ is
\[
T_{\mathcal{I}, s}(w_{\mathcal{I}}, w_{[q] \backslash \mathcal{I}}) 
=
\begin{cases}
\prod_{i: s_i = 1} \lambda^{\sigma(w)_{i+1}-\sigma(w)_i} & \text{if } w \in W_{\mathcal{I}}\\
0 & \text{ otherwise}.
\end{cases}
\]
Because 
\[
\bigcup_{\mathcal{I} \in \binom{[q]}{q/2}} W_{\mathcal{I}} = [n]^q,
\]
Eq.~\eqref{eq:monomintermediate} can be bounded above by
\begin{align*}
\sum_{s \in S_{q-1}}\sum_{\mathcal{I} \in \binom{[q]}{q/2}} \sum_{w \in W_{\mathcal{I}}} a_{w_1}a_{w_2}\cdots a_{w_q}&\left(\prod_{i: s_i = 1} \lambda^{\sigma(w)_{i+1}-\sigma(w)_i} \right)
=
\sum_{s \in S_{q-1}}\sum_{\mathcal{I} \in \binom{[q]}{q/2}} \left\langle a^{\otimes q/2}, T_{\mathcal{I}, s} a^{\otimes q/2}\right\rangle_{\ell_2} \\
&\leq |S_{q-1}|\binom{q}{q/2}\max_{s \in S_{q-1}, \mathcal{I} \in \binom{[q]}{q/2}}\|T_{\mathcal{I}, s}\|_{\ell_2 \rightarrow \ell_2}\|a\|_{\ell_2}^q,
\end{align*}
where 
$a^{\otimes q/2} \in \R^{n^{q/2}}$ is the vector such that $a^{\otimes q/2}_{i_1, \ldots, i_{q/2}} = a_{i_1}a_{i_2}\cdots a_{i_{q/2}}$ for $i \in [n]^{q/2}$ and thus $\|a^{\otimes q/2}\|_{\ell_2} = \|a\|_{\ell_2}^{q/2}$.
Both $|S_{q-1}|$ and $\binom{q}{q/2}$ are each bounded above by $2^q$.
Thus by Claim~\ref{thm:mlrt}, it is enough to show that
\[
\|T_{\mathcal{I}, s}\|_{\ell_1 \rightarrow \ell_1}, \|T_{\mathcal{I}, s}\|_{\ell_{\infty} \rightarrow \ell_{\infty}} \leq (q/2)!\left(\frac{1}{1-\lambda}\right)^{q/2}.
\]
We show this for $\|T_{\mathcal{I}, s}\|_{\ell_{\infty} \rightarrow \ell_{\infty}}$; the proof for $\|T_{\mathcal{I}, s}\|_{\ell_1 \rightarrow \ell_1}$ is similar.

Because the entries of $T$ are positive, $\|T_{\mathcal{I}, s}\|_{\ell_{\infty} \rightarrow \ell_{\infty}}$ is just the largest row sum of $T_{\mathcal{I}, s}$.
Without loss of generality, assume that $\mathcal{I} = \{1, 3, 5, \ldots, q-1\}$.
Then the sum of the entries of the row corresponding to $w_{\mathcal{I}} = (w_1, w_3, w_5, \ldots, w_{q-1})$ is
\begin{align*}
\sum_{w_2, w_4, \ldots, w_q \st w \in W_{\mathcal{I}}} T_{\mathcal{I}, s}(w_{\mathcal{I}}, w_{[q] \backslash \mathcal{I}}) 
&\leq (q/2)!\sum_{w_2 = \sigma(w)_1}^{\sigma(w)_3} \sum_{w_4 = \sigma(w)_3}^{\sigma(w)_5} \cdots \sum_{w_q = \sigma(w)_{q-1}}^{n} \prod_{i: s_i = 1}\lambda^{\sigma(w)_{i+1}-\sigma(w)_i} \\
&\leq (q/2)! \left(\frac{1}{1-\lambda}\right)^{q/2},
\end{align*}
as desired.
The first inequality follows from the fact that $w \in W_{\mathcal{I}}$ and $w_1, w_3, w_5, \ldots, w_{q-1}$ determine $\sigma(w)_1, \sigma(w)_3, \sigma(w)_5, \ldots, \sigma(w)_{q-1}$ exactly,
and that there are at most $(q/2)!$ possible orderings of $w_2, w_4, \ldots, w_q$.
The second inequality follows from the definition of $S_{q-1}$, which implies that for every positive even integer $k \leq q$, either $s_{k-1} = 1$ or $s_{k} = 1$, along with the formula for the sum of an infinite geometric series.
\end{proof}

Finally, Theorem~\ref{cor:tailbound} follows by considering the moment generating function and applying Markov's inequality.

\begin{proof}[ of Theorem~\ref{cor:tailbound}]
If $\lambda \geq 1$ or if $u \leq 8/\sqrt{1-\lambda}$, the theorem holds trivially as the right-hand side is greater than $1$.

Otherwise, we start by bounding the moment generating function.
Let $\theta = (1-\lambda)u/(32(a_1^2+\cdots+a_n^2)^{1/2})$
By Theorem~\ref{lem:partition} and keeping in mind that by Jensen's inequality, odd moments are bounded above by even moments,
\begin{align*}
\E\left[\exp(\theta(f_1(Y_1)+\cdots+f_n(Y_n)))\right]
&=
\sum_{q=0}^{\infty} \frac{\E[\theta(f_1(Y_1)+\cdots+f_n(Y_n))^q]}{q!} \\
&\leq 1+\sum_{q=1}^{\infty} \frac{(1-\lambda)^{(2q-1)/2} u^{2q-1} q!}{8^{2q-1}(2q-1)!}+\frac{(1-\lambda)^{q} u^{2q} q!}{8^{2q} (2q)!} \\
&\leq 2 \sum_{q=0}^{\infty} \frac{(1-\lambda)^q u^{2q}}{8^{2q} q!} \\
&= 2\exp\left(u^2(1-\lambda)/64\right).
\end{align*}
By Markov's inequality, 
\begin{align*}
\Pr\Biggr[f_1(Y_1)+\cdots+f_n(Y_n) &\geq u\Bigg(\sum_{i=1}^n a_i^2\Bigg)^{1/2} \Biggr] \\
&=
\Pr\Biggr[\exp(\theta (f_1(Y_1)+\cdots+f_n(Y_n))) \geq \exp\Biggr(\theta u\Biggr(\sum_{i=1}^n a_i^2\Biggr)^{1/2} \Biggr)\Biggr] \\
&\leq 
\frac{\E\left[\exp(\theta (f_1(Y_1)+\cdots+f_n(Y_n)))\right]}{ \exp\left(\theta u\left(\sum_{i=1}^n a_i^2\right)^{1/2} \right)} \\
&\leq
2\exp\left(u^2(1-\lambda)/64 - u^2(1-\lambda)/32\right) \\
&= 2\exp\left(-u^2(1-\lambda)/64\right)
\end{align*}
The final bound follows by doing the same for the left tail, and noting that if $u \geq 8/\sqrt{1-\lambda}$, either $4\exp(-u^2(1-\lambda)/64) \leq 2\exp(-u^2(1-\lambda)/(64e))$, or $2\exp(-u^2(1-\lambda)/(64e) \geq 1$.

\end{proof}

We note that it is possible to obtain stronger tail bounds that improve on the constant factor by optimizing some of the calculations above, but we will not do so here.

\section{Extension to vector--valued random variables}

To prove Theorem~\ref{thm:app} we use the techniques of Talagrand's generic chaining.
These techniques apply to random variables that satisfy the ``increment condition," which we define below.
\begin{definition}
A metric space $(T, d)$ and process $(Z_t)_{t \in T}$ satisfies the increment condition if for all $u$ and all $s, t \in T$,
\[
\Pr[|Z_s-Z_t| \geq u] \leq 2\exp\left(-\frac{u^2}{2d(s, t)^2}\right).
\]
\end{definition}

When $(Z_t)_{t \in T}$ is a gaussian process, that is $Z_t$ is gaussian for all $t \in T$, we can equip $T$ with the canonical distance, $d(s, t) = \E[(Z_s-Z_t)^2]^{1/2}$.

Theorem~\ref{cor:tailbound} essentially states that for a a Markov chain $\{Y_i\}_{i=1}^{\infty}$ and functions $f_1, \ldots, f_n: [N] \rightarrow [-1, 1]$ with $\E[f_i(Y_i)] = 0$, the process $(Z_t)_{t \in T}$ defined by $Z_t = (f_1(Y_1)t_1, \ldots, f_n(Y_n)t_n)$ for $T = \R^n$ satisfies the increment condition if the associated distance is $\sqrt{32e/(1-\lambda)}$ times the Euclidean distance.

We also define the $\gamma_2$ functional.
\begin{definition}
\[
\gamma_2(T, d) = \inf \sup_{t \in T} \sum_{i=0}^{\infty} 2^{i/2} \min_{t' \in T_i} d(t, t'),
\]
where the infimum is taken over all sequences of subsets $T_0 \subseteq T_1\subseteq \cdots \subseteq T$ such that $|T_0| = 1$ and $|T_i| \leq 2^{2^{i}}$ for $i \geq 1$.
\end{definition}


The majorizing measures theorem, due to Talagrand~\cite{T87} (see also Theorem 2.4.1 in~\cite{T14}), gives bounds on the expected value of $\sup_{t \in T} Z_t$, where $(Z_t)_{t\in T}$ is a gaussian process, in terms of $\gamma_2(T, d)$ where $d$ is the canonical distance.
We state the theorem below.

\begin{theorem}[Talagrand's majorizing measures theorem]\label{thm:mm}
For some universal constant $C$, and for every gaussian process $(Z_t)_{t \in T}$,
\begin{equation*}\label{eq:mm}
\frac{1}{C}\gamma_2(T, d)\leq 
\mathbb{E}\left[\sup_{t \in T}Z_t\right]
\leq
{C}\gamma_2(T, d),
\end{equation*}
where $d(s, t) = \E[(Z_s-Z_t)^2]^{1/2}$.
\end{theorem}


We also use the following tail bound for any process that satisfies the increment condition, which is given as Theorem 2.2.27 in~\cite{T14}.

\begin{theorem}\label{thm:taltail}
If the process $(Z_t)$ satisfies the increment condition, then for $u > 0$,
Then,
\[
\Pr\left[\sup_{s, t \in T} |X_s-X_t| \geq L \gamma_2(T, d) + u L \sup_{t_1, t_2 \in T} d(t_1, t_2)\right] \leq L\exp(-u^2).
\]
\end{theorem}

We now describe how to select $T$ to apply the above tools to the setting of Theorem~\ref{thm:app}.
Let $(X, \|\cdot\|)$ be a Banach space, and let $(X^*, \|\cdot\|_*)$ be the dual space of $X$ with closed unit ball $B^*$.
Recall that for $x \in X$,
\[
\|x\| = \sup_{x^* \in B^*} |\langle x^*, x\rangle|.
\]
(see for instance, Theorem 4.3 in~\cite{R91}).
For fixed $X_1, \ldots, X_n \in X$, let $T \subset \R^{n}$ be the set of points,
\begin{equation}
T = \left\{(\langle x^*, X_1\rangle, \langle x^*, X_2 \rangle, \ldots, \langle x^*, X_n \rangle) \st x^* \in B^* \right\}.
\label{eq:defpoints}
\end{equation}
Note that $T$ is symmetric, as for every $x^* \in B^*$, we also have $-x^* \in B^*$.
It follows that
\begin{equation}\label{eq:propofT}
\left\|f_1X_1+\cdots+f_nX_n\right\|
=
\sup_{t \in T} \langle f, t \rangle.
\end{equation}

Finally, we prove Theorem~\ref{thm:app}.

\begin{proof}[ of Theorem~\ref{thm:app}]
Consider the metric space $(T, d)$ where $T$ is as constructed in Eq.~\eqref{eq:defpoints} and $d(s, t) = \sqrt{32e/(1-\lambda)}\|s-t\|_{\ell_2}$.
Then by Theorem~\ref{cor:tailbound}, the process $(Z_t)_{t \in T}$ defined by $Z_t = (f_1(Y_1)t_1, \ldots, f_n(Y_n)t_n)$ satisfies the increment condition.

Additionally, consider the Gaussian process $(Z'_t)_{t \in T}$ on the metric space $(T, d')$, so that $Z_t = g_1 t_1 +\cdots + g_n t_n$ for independent standard Gaussian variables $g_1, \ldots, g_n$ and $d' = \E[(Z_s-Z_t)^2]^{1/2}$.
Then by Theorem~\ref{thm:mm},
\[
\gamma_2(T, d) = \sqrt{\frac{32e}{1-\lambda}}\gamma_2(T, d') \leq \frac{C}{1-\lambda} \mathbb{E}\left[\sup_{t \in T}Z_t'\right] 
\]

The theorem then follows from Theorem~\ref{thm:taltail} the observation that $\sup_{s, t} |Z_s-Z_t| = 2\sup_{t} Z_t$ as $T$ is symmetric, and Eq.~\eqref{eq:propofT}.
\end{proof}

%

\subsection{Comparison to matrices with independent entries}

We prove Corollary~\ref{cor:apptobvh}, which follows from a straightforward application of Theorem~\ref{thm:app}.

In order to apply Theorem~\ref{thm:app}, we need a bound on $\E[\|X'\|_{S_{\infty}}]$ when $X'$ is the random symmetric matrix whose entries are
\[
X'_{i, j} = 
\begin{cases}
g_{i, j}b_{i, j} & \text{ if } (i, j) \in \mathcal{I}\\
g_{j, i} b_{i, j} & \text{ otherwise}
\end{cases} 
\]
where $g_{i, j} \sim \mathcal{N}(0, 1)$ are independent standard Gaussian random variables (rather than Rademacher random variables, as in Eq.~\eqref{eq:bvh}).
This setting was also discussed in~\cite{BvH16} in which it was shown that
\begin{equation}\label{eq:bvhgaussian}
\E[\|X'\|_{S_{\infty}}] \leq C (\sigma + \sigma_*\sqrt{\log d}),
\end{equation}
where $\sigma$ and $\sigma_*$ are defined as in Eq.~\eqref{eq:sigmadef}.

\begin{proof}[ of Corollary~\ref{cor:apptobvh}]
Let $X'$ be the random matrix defined above.
Then by Theorem~\ref{thm:app} and Eq.~\eqref{eq:bvhgaussian},
\[
\E[\|X\|_{S_{\infty}}] \leq \frac{C}{\sqrt{1-\lambda}} \E[\|X'\|_{S_{\infty}}] \leq \frac{C'}{\sqrt{1-\lambda}} (\sigma + \sigma_*\sqrt{\log d})
\]
Finally, because $|f(v)| \leq 1$ for all $v \in [N]$ and $B$ has positive entries, it follows that $\|X\|_{S_{\infty}} \leq \|B\|_{S_{\infty}}$, always.
\end{proof}

\subsection*{Acknowledgments} I would like to thank Oded Regev, Noah Stephens-Davidowitz, and the anonymous referees for their valuable comments.
I would also like to thank the anonymous referees for pointing out that the proof also applies to non-reversible Markov chains.

\bibliographystyle{alphaabbrv}
\bibliography{expanderhoeffding}

\appendix

\section{}

In this section, we give the tools needed to prove Lemma~\ref{lem:monomial}.
They are either taken directly from~\cite{RR17} (which is based on techniques used in~\cite{NaorRR17}), or are straightforward adaptations.

\begin{claim}\label{clm:jbetween}
For all $k \ge 1$, matrices $R_1,\ldots,R_k \in \R^{N \times N}$, and distributions $\pi$ over $[N]$
\[
\left \langle \mathbf{1}, R_1 E_{\pi} R_2 E_{\pi} \cdots E_{\pi} R_k \mathbf{1} \right\rangle_{L_2(\pi)}
=
\prod_{i=1}^k \langle \mathbf{1}, R_i\mathbf{1} \rangle_{L_2(\pi)}
\le \prod_{i=1}^k \|R_i \mathbf{1}\|_{L_1(\pi)} \; .
\]
\end{claim}

\begin{claim}\label{clm:claimcombo}
For all $k \ge 1$, vectors $u_1, \ldots, u_{k} \in \R^N$, $U_i = \diag(u_i)$ for all $i$, distributions $\pi$ over $[N]$
and matrices $T_1, \ldots, T_{k-1} \in \R^{N \times N}$,
\[
\left\|U_1 T_1 U_2 T_2 \cdots T_{k-1} U_k \mathbf{1} \right\|_{L_1(\pi)} 
\leq 
\|u_k\|_{L_\infty(\pi)}\prod_{i=1}^{k-1} \|u_i\|_{L_\infty(\pi)} \|T_i\|_{L_2(\pi)} \; .
\]
\end{claim}

\begin{proof}
By Jensen's inequality, the right-hand side is bounded above by
\[
\left\|U_1 T_1 U_2 T_2 \cdots T_{k-1} U_k \mathbf{1} \right\|_{L_2(\pi)} 
\]
and the claim follows by the definition of operator norm, and the fact that $\|U_i\|_{L_2(\pi) \rightarrow L_2(\pi)} = \|u_i\|_{L_{\infty}(\pi)}$.
\end{proof}


\begin{lemma}\label{lem:holderapplication}
Let $k \ge 1$ be an integer.
Let $S_k \subset \{0, 1\}^k$ be the subset of $\{0, 1\}^k$ of vectors $s$ with no two consecutive $0$s and so that $s_1, s_k = 1$. 
Let $\pi$ be a distribution over $[N]$, let $u_1, \ldots, u_{k+1} \in \R^{N}$ be $N$-dimensional vectors such that $u_i^t\pi = 0$ for all $i$, and let $U_i = \diag(u_i)$ for all $i$.
Finally, let $T_1, \ldots, T_k \in \R^{N \times N}$.
Then,
\begin{multline}\label{eq:splittingj}
\left|\left\langle \mathbf{1}, U_1 (T_1+E_{\pi}) U_2 (T_2+E_{\pi}) U_3\cdots U_k (T_k+E_{\pi}) U_{k+1} \mathbf{1} \right\rangle_{L_2(\pi)}\right|
\leq \\
\|u_1\|_{L_{\infty}(\pi)}\|u_2\|_{L_{\infty}(\pi)}\cdots\|u_{k+1}\|_{L_{\infty}(\pi)} \sum_{s \in S_{k}}\prod_{j: s_j = 1}\|T_j\|_{L_2(\pi) \rightarrow L_2(\pi)}\; .
\end{multline}
\end{lemma}
\begin{proof}
For $j=1,\ldots,k$, let $T_{j, 0} = E_{\pi}$ and $T_{j, 1} = T_j$. 
Then using the triangle inequality, the left-hand side of~\eqref{eq:splittingj} is at most
\begin{align}
\sum_{s \in \{0, 1\}^k} \left|\left\langle \mathbf{1}, \left(\prod_{j=1}^k U_jT_{j, s_j}\right) U_{k+1} \mathbf{1}\right\rangle_{L_2(\pi)}\right|
=
\sum_{s \in S_k} \left|\left\langle \mathbf{1}, \left(\prod_{j=1}^k U_jT_{j, s_j}\right) U_{k+1} \mathbf{1}\right\rangle_{L_2(\pi)}\right|, \label{eq:mainclaiminmono}
\end{align} 
since the terms corresponding to vectors $s$ with two consecutive zeros or with $s_k = 0$ are equal to $0$
because in these cases the term $E_{\pi}U_jE_{\pi}=0$ (or $E_{\pi}U_{k+1}\mathbf{1}=0$) appears.
Additionally, terms corresponding to vectors $s$ with $s_1 = 0$ are equal to $0$, as $\langle \mathbf{1}, U_1E_{\pi}v\rangle_{L_2(\pi)} = 0$ for all $v \in \R^{N}$.

Fix an $s \in S_k$, and let $r_1, r_2, \ldots, r_\ell$ be the indices of $s$ that are $0$.  
By Claim~\ref{clm:jbetween}, the term corresponding to $s$ in Eq.~\eqref{eq:mainclaiminmono}
is at most
\begin{multline*}
\|U_1T_{1}U T_{2}\cdots T_{r_1-1} U_{r_1} \mathbf{1} \|_{L_1(\pi)} \cdot
\|U_{r_1+1}T_{r_1+1}U_{r_1+2} T_{r_1+2}\cdots T_{r_2-1} U_{r_2} \mathbf{1} \|_{L_1(\pi)}
\cdots \\
\|U_{r_{\ell}+1}T_{r_{\ell}+1}U_{r_{\ell}+2} T_{r_{\ell}+2}\cdots T_{k} U_{k+1} \mathbf{1} \|_{L_1(\pi)}  \; .
\end{multline*}
The claim now follows by applying Claim~\ref{clm:claimcombo}.
\end{proof}

\end{document}